\documentclass[12pt]{amsart}
\usepackage{amsmath, amsthm, mathabx,amscd, amsfonts, amssymb, hyperref, mathrsfs, textcomp, epsfig, a4wide, graphicx, IEEEtrantools, tikz, verbatim}

\newtheorem{thm}{Theorem}
\newtheorem{prop}{Proposition}

\theoremstyle{remark}
\newtheorem*{remark}{Remark}
\begin{document}

\title{Monopole Floer homology and the spectral geometry of three-manifolds} 

\author{Francesco Lin}
\address{Department of Mathematics, Princeton University and School of Mathematics, Institute for Advanced Study} 
\email{fl4@math.princeton.edu}

\begin{abstract}
We refine some classical estimates in Seiberg-Witten theory, and discuss an application to the spectral geometry of three-manifolds.
In particular, we show that on a rational homology three-sphere $Y$, for any Riemannian metric the first eigenvalue of the laplacian on coexact one-forms is bounded above explicitly in terms of the Ricci curvature, provided that $Y$ is not an $L$-space (in the sense of Floer homology). The latter is a purely topological condition, and holds in a variety of examples. Performing the analogous refinement in the case of manifolds with $b_1>0$, we obtain a gauge-theoretic proof of an inequality of Brock and Dunfield relating the Thurston and $L^2$ norms of hyperbolic three-manifolds, first proved using minimal surfaces.
\end{abstract}

\maketitle

We discuss an application of gauge theory and Floer homology to spectral geometry, and in particular to the study of the spectrum of the Hodge laplacian acting on differential forms on a compact Riemannian manifold. A classical upper bound can be provided in terms of the sectional curvatures, their covariant derivatives and the injectivity radius (\cite{Dod}), while some lower bounds can be exhibited in special cases using the Mayer-Vietoris arguments introduced in \cite{McG} (see also \cite{GP}). In the present paper we focus on three-manifolds. By the Hodge decomposition, in this case the spectrum of the Laplacian on forms is determined by the one on functions (which is somewhat well-understood) and the one on coexact $1$-forms.
Given a rational homology sphere $Y$, we will provide an upper bound on the lowest eigenvalue $\lambda^*_1$ on coexact $1$-forms purely in terms of the Ricci curvature, provided an extra topological assumption (which is gauge-theoretical in nature). In the following, we denote by $\tilde{s}(p)$ the sum of the two least eigenvalues of the Ricci curvature at the point $p$.
\begin{thm}\label{spectrum}
Let $Y$ be a rational homology sphere which is not an $L$-space. Then for every Riemannian metric on $Y$ the upper bound
\begin{equation}\label{top}
\lambda^*_1\leq -\mathrm{inf}_{p\in Y}\tilde{s}(p)
\end{equation}
holds. In particular, a lower bound on the Ricci curvature implies an upper bound on $\lambda_1^*$.
\end{thm}

An $L$-space $Y$ is a rational homology sphere $Y$ for which the reduced monopole Floer homology $\mathit{HM}(Y)$ vanishes (see \cite{KM}). This condition is a topological invariant. An alternative definition is the vanishing of the reduced Heegaard Floer homology $HF^{red}(Y)$ (\cite{OS1}): these conditions are equivalent via the isomorphism between the two theories (see \cite{HFHM1}, \cite{CGH} and subsequent papers). Examples of $L$-spaces include spherical space forms (\cite{KM}) and branched double covers of alternating knots (\cite{OSBr}). In general, the condition of being an $L$-space is quite well understood, and algorithmically computable (\cite{SW}). Among rational homology spheres which are \textit{not} $L$-spaces (so that the main result of the paper applies to them) we have the following examples:
\begin{itemize}
\item any \textit{integral} homology sphere obtained by a surgery on a knot in $S^3$ other than $S^3$ itself and the Poincar\'e homology sphere (\cite{Ghi});
\item any rational homology sphere obtained by surgery on a knot $K$ is $S^3$ such that the Alexander polynomial of $K$ has a coefficient different than $\pm1$ (\cite{OSL}). More in general, there are many restrictions on knots which admit an $L$-space surgery (see for example \cite{Hed}).
\item any rational homology sphere that admits a taut foliation (\cite{KMOS}).
\end{itemize}
An intriguing conjecture (\cite{BGW}) states that a rational homology sphere is an $L$-space if and only if it does not admit a taut foliation, if and only if its fundamental group is not left-orderable. This has been recently verified for graph manifolds (\cite{HRRW}).

\vspace{0.5cm}
To prove the main result, we refine some well-known estimates for the solutions of the Seiberg-Witten equations. This is inspired by the estimates in the four dimensional case involving the self-dual Weyl curvature discussed in \cite{LeB}. The key idea in our case is to exploit the classical Bochner formula connecting the Hodge Laplacian and the Bochner Laplacian on $1$-forms in terms of the Ricci curvature.
\begin{proof}[Proof of Theorem \ref{spectrum}]
We follow the conventions of \cite{KM}. A sufficient condition for a rational homology sphere $Y$ to be an $L$-space is the existence of a pair of metric and perturbation which is admissible and for which the Seiberg-Witten equations
\begin{align*}
\frac{1}{2}\rho(F_{B^t})-(\Psi\Psi^*)_0&=0\\
D_B\Psi&=0
\end{align*}
do not have irreducible solutions. Let $(B,\Psi)$ be any solution. Then
\begin{equation}\label{lap}
\Delta|\Psi|^2=2\langle \Psi, \nabla_B^*\nabla_B \Psi\rangle -2 |\nabla_B\Psi |^2=-|\Psi|^4-\frac{1}{2}s|\Psi|^2-2|\nabla_B\Psi|^2,
\end{equation}
where we used the Weitzenb\"ock identity

\begin{equation*}
D_B^2\Psi=\nabla_B^*\nabla_B\Psi+\frac{1}{2}\rho(F_{B^t})\cdot \Psi+\frac{s}{4}\Psi.
\end{equation*}
We can now multiply (\ref{lap}) by $|\Psi|^2$, integrate over the manifold, and obtain by Green's identity
\begin{equation}\label{green}
\int |\Psi|^6+\frac{1}{2}s|\Psi|^4+2|\Psi|^2|\nabla_B\Psi|^2=-\int |\Psi|^2\Delta|\Psi|^2=-\int \lvert d |\Psi|^2\rvert^2\leq 0.
\end{equation}
The key idea is to get a better bound on the third term on the left hand side. To do this, recall that for a $1$-form $\xi$ the classical Bochner identity
\begin{equation*}
(d+d^*)^2\xi=\nabla^*\nabla\xi +\mathrm{Ric}(\xi,\cdot)
\end{equation*}
holds, see for example \cite{Bes}. Suppose now that $\xi$ is coclosed. Integrating by parts, we obtain the inequality
\begin{equation}\label{boch}
\int |\nabla\xi|^2=\int|(d+d^*)\xi|^2+\int -\mathrm{Ric}(\xi,\xi)\geq\int(\lambda^{*}_1- m)|\xi|^2,
\end{equation}
where $m(p)$ is the maximum eigenvalue of the Ricci curvature at $p$, and we consider the variational definition of the first eigenvalue: as $b_1(Y)=0$, there are not non-trivial harmonic $1$-forms, so that we have
\begin{equation}\label{ineqlambda}
\int|(d+d^*)\xi|^2\geq \lambda^*_1\int|\xi|^2.
\end{equation}
We now apply this last inequality to the $1$-form $\xi=\rho^{-1}(\Psi\Psi^*)_0$, which is coclosed because its Hodge star is a multiple of the curvature (recall that on a three-manifold, for any $1$-form we have $\rho(\alpha)=-\rho(\ast\alpha)$). For this $1$-form, recalling that we are using the inner product on $i\mathfrak{su}(2)$ given by $\mathrm{tr}(a^*b)/2$ (which makes $\rho$ an isometry), we have
\begin{equation*}
|\nabla\xi|^2\leq |\Psi|^2|\nabla_B\Psi|^2,\quad |\xi|^2=\frac{1}{4}|\Psi|^4,
\end{equation*}
hence substituting in (\ref{green}) we obtain
\begin{equation*}
\int|\Psi|^6+\frac{1}{2}(\lambda_1^{*}+\tilde{s})|\Psi|^4\leq 0.
\end{equation*}
Here be definition $\tilde{s}=s-m$. Now, if we assume that $\lambda_1^{*}+\tilde{s}$ is non-negative, $\Psi$ is forced to be identically zero, hence the Seiberg-Witten equations do not have irreducible solutions. As the quantity $\lambda_1^{*}+\tilde{s}$ is by assumption everywhere strictly positive, the same result holds for a small admissible perturbation of the equations (as those constructed in \cite{KM}), so that $Y$ is an $L$-space.
\end{proof}
\begin{remark}Of course (\ref{top}) the does not hold for spherical space forms, but one can also construct an example of Riemannian three-manifold with $\mathrm{inf}\tilde{s}<0$ for which (\ref{top}) does not hold. For example, the Hantzsche-Wendt manifold is the unique rational homology three-sphere admitting a flat metric. We can choose a small perturbation of the metric for which $\mathrm{inf}\tilde{s}<0$. As $\lambda_1^*$ is a Lipshitz function of the metric, if the perturbation is small enough (\ref{top}) will still be false.
\end{remark}

\vspace{1cm}

The present paper is motivated by the following interesting result in \cite{BD}, which was in turn inspired by \cite{BSV}. Recall that the first cohomology $H^1(Y;\mathbb{R})$ of an oriented hyperbolic three-manifold $Y$ comes with two natural norms: the Thurston norm $\|\cdot \|_{Th}$ (\cite{Thu}), and the harmonic norm with respect to the hyperbolic metric $\|\cdot\|_{L^2}$.

\begin{thm}\label{hypnorm}[Theorem $1.3$ of \cite{BD}]Given a closed oriented hyperbolic three-manifold $Y$, the inequality
\begin{equation}\label{Thvol}
\frac{\pi}{\sqrt{\mathrm{vol}(Y)}}\|\cdot \|_{Th}\leq \|\cdot\|_{L^2}
\end{equation}
holds.
\end{thm}

The authors also show that the inequality is qualitatively sharp (see Theorem $1.5$ in \cite{BD}). Their proof relies on the theory of minimal surfaces in hyperbolic three-manifolds, and in particular an idea of Uhlenbeck \cite{Uhl}. The authors also point out that a weaker inequality is a direct consequence of the main result of \cite{KMs}, which asserts that for a closed oriented irreducible three-manifold $Y$, for any class $\alpha\in H^2(Y;\mathbb{R})$ the identity
\begin{equation}\label{sup}
|\alpha|=4\pi\cdot\mathrm{sup}_h \|\alpha\|_{L^2(h)}/\|s_h\|_{L^2(h)}
\end{equation}
holds. Here $|\cdot|$ denotes the dual Thurston norm, $h$ varies along all Riemannian metrics on $Y$ and $s_h$ is the scalar curvature. Passing to duals, this is equivalent to the fact that for a given $\phi\in H^1(Y;\mathbb{R})$ we have the identity
\begin{equation*}
4\pi \|\phi\|_{Th}=\mathrm{inf}_h \|\phi\|_{L^2(h)}\|s_h\|_{L^2(h)}.
\end{equation*}
In our case of interest, by taking $Y$ to be a the hyperbolic metric, which has scalar curvature $-6$, we obtain the bound
\begin{equation*}
\frac{2\pi}{3\sqrt{\mathrm{vol}(Y)}}\|\cdot \|_{Th}\leq \|\cdot\|_{L^2},
\end{equation*}
which is weaker than the one in Theorem \ref{hypnorm}. Using the same estimates we introduced above, we can prove the following.

\begin{prop}\label{new}
For a closed oriented irreducible three-manifold $Y$, for any class $\alpha\in H^2(Y;\mathbb{R})$ the identity
\begin{equation}\label{sup1}
|\alpha|=4\pi\cdot\mathrm{sup}_h \|\alpha\|_{L^2(h)}/\left(\mathrm{vol}_h^{1/6}\|\tilde{s}_h\|_{L^3(h)}\right).
\end{equation}
holds.
\end{prop}
Passing to duals, we obtain as above
\begin{equation*}
4\pi \cdot\|\phi\|_{Th}=\mathrm{inf}_h \|\phi\|_{L^2(h)}\left(\mathrm{vol}_h^{1/6}\|\tilde{s}_h\|_{L^3(h)}\right)
\end{equation*}
so that we obtain Theorem \ref{hypnorm}, as for a hyperbolic metric $\tilde{s}=-4$. While the proof in \cite{BD} uses extensively the sectional curvatures, our approach only relies on the Ricci curvature.

\begin{proof}[Proof of Proposition \ref{new}]
Following \cite{KMs}, define the polytopes in the cohomology
\begin{align*}
P_1&=\{\text{convex hull of monopole classes}\}\\
P_2&=\{\text{unit ball of }4\pi\cdot\mathrm{sup}_h \|\alpha\|_{L^2(h)}/(\mathrm{vol}_h^{1/6}\|\tilde{s}_h\|_{L^3(h)})\}\\
P_3&=\{\text{unit ball of the dual Thurston norm}\}.
\end{align*}
We claim that $P_1\subset P_2\subset P_3$: the key result of \cite{KMs} (which builds on deep work of Gabai, Thurston, Eliashberg and Taubes) shows that $P_3\subset P_1$, so that these three polytopes coincide and the result follows. Consider the first containment $P_1\subset P_2$. We proceed as in the proof of Theorem \ref{spectrum}, the only difference being that as $b_1(Y)>0$ the inequality (\ref{ineqlambda}) does not hold. Nevertheless, the quantity on its left hand side is still clearly non-negative, so that the inequality
\begin{equation*}
\int|\Psi|^6+\frac{1}{2}\tilde{s}|\Psi|^4\leq 0
\end{equation*}
holds. By applying H\"older inequality we obtain
\begin{equation*}
\int|\Psi|^6\leq -\frac{1}{2}\tilde{s}|\Psi|^4\leq \left(\int|\tilde{s}/2|^3\right)^{1/3}\left(\int|\Psi|^6\right)^{2/3}
\end{equation*}
so that (as $\Psi\neq 0$) we have
\begin{equation*}
\int|\tilde{s}/2|^3\geq \int|\Psi|^6.
\end{equation*}
Again H\"older inequality and the Seiberg-Witten equations imply that
\begin{equation*}
\mathrm{vol}(Y)^{1/3}\left(\int|\Psi|^6\right)^{2/3}\geq \int|\Psi|^4=\int |F_{B^t}|^2,
\end{equation*}
so that putting the pieces together
\begin{equation*}
\frac{1}{4}\mathrm{vol}(Y)^{1/3}\|\tilde{s}^2\|^2_{L^3(h)}\geq  \int |F_{B^t}|^2\geq 4\pi^2 \|\alpha\|^2_{L^2(h)}
\end{equation*}
as $F_{B^t}$ represents the class $-2\pi i\alpha$.
\par
Finally, the proof of the cointainment $P_2\subset P_3$ is obtained as in \cite{KMs} by the following neck stretching argument. Consider an embedded surface of genus at least $2$, and suppose the metric $h$ has a neck of the form $\Sigma\times[0,R]$ where $\Sigma$ has constant negative curvature and unit area, and is independent of $R$ outside a the neck. On the neck, as the metric is a product, we have $\tilde{s}=s$, so that
\begin{equation*}
\|\tilde{s}_h\|_{L^3(h)}= 4\pi(2g-2)R^{1/3}+O(1),\quad \mathrm{vol}_h=R+O(1).
\end{equation*}
Th argument then is concluded as in \cite{KMs}: any $2$-form $\omega$ representing $\alpha$ must satisfy
\begin{equation*}
\|\omega\|_h\geq R^{1/2}\langle \alpha,[\Sigma]\rangle,
\end{equation*}
hence 
\begin{equation*}
\langle \alpha,[\Sigma]\rangle/(2g-2)\leq\mathrm{sup}_h  \|\alpha\|_{L^2(h)}/\left(\mathrm{vol}_h^{1/6}\|\tilde{s}_h\|_{L^3(h)}\right).
\end{equation*}
The case of an embedded torus is handled similarly.
\end{proof}

\vspace{0.3cm}
We conclude by discussing the extremal cases. In \cite{IY} it is shown that a metric for which the identity (\ref{sup}) is realized is very constrained: among the others, the geometry is forced to be $\mathbb{R}\times\mathbb{H}$, so that in particular $Y$ is Seifert. This follows from the fact that the curvature $F_{B^t}$ is parallel. In our case, we can instead conclude if a manifold $Y$ admits a metric for which the identity (\ref{sup}) is realized must fiber over the circle. Indeed, if we are in the extremal case, inequality (\ref{boch}) implies that the curvature $F_{B^t}$ is harmonic, so that in particular $\ast F_{B^t}$ is a closed form. Inequality (\ref{green}) implies that $|\Psi|$ is a non-zero constant, so that the by the Seiberg-Witten equations $\ast F_{B^t}$ is never zero. As $\ast F_{B^t}$ is (up to constants) the Poincar\'e dual to an integral form, by integrating it we obtain the required fibration to $S^1$.

\vspace{0.5cm}
\textit{Acknowledgements. }The author would like to thank Otis Chodosh, Tom Mrowka and Peter Sarnak for some nice conversations. This work was supported by the the Shing-Shen Chern Membership Fund and the IAS Fund for Math.

\vspace{0.3cm}
\bibliographystyle{alpha}
\bibliography{biblio}

\end{document}